\documentclass[12pt]{article}
\usepackage[margin=35mm]{geometry}
\usepackage{amsmath}
\usepackage{amsfonts}
\usepackage{amssymb}
\usepackage{amsthm}
\usepackage{graphicx}
\usepackage{upquote}
\usepackage{verbatim}
\usepackage{tikz}
\immediate\write18{texcount -tex -sum  \jobname.tex > \jobname.wordcount.tex}
\usepackage{xcolor}
\usepackage[colorlinks=true,
            linkcolor=blue!100!black,   
            citecolor=blue!100!black,  
            urlcolor=blue!100!black]{hyperref}
\newtheorem{theorem}{Theorem}
\newtheorem{lemma}{Lemma}
\newtheorem{observation}{Observation}

\providecommand{\keywords}[1]
{
  \small	
  \textbf{\textit{Keywords---}} #1
}

\title{Ramsey numbers $R(K_{1,4},K_{m,n})$ for small $m$}
\author{Khaidar Abdullin$^{1}$, Dmitrii Taletskii$^{1}$\\
\small
$^1$ National Research University Higher School of Economics, \\
\small
Bolshaya Pechyorskaya ul. 25/12, Nizhny Novgorod, 603155 Russia\\%
    \small
    e-mail: abdullinhaidar@gmail.com, dmitalmail@gmail.com
}
\date{} 

\begin{document}
\maketitle

\begin{abstract}
We determine all currently unknown values of the generalized two-color Ramsey number $R(K_{1,4},K_{m,n})$ for $\min(m,n) \leq 8$. We also prove that the number $R(K_{1,4},K_{m,m})$ equals $26, 29, 31, 33$ for $m = 9, 10, 11, 12$ respectively.  
\end{abstract} 


\section{Introduction}
\label{sec:introduction}

For finite simple graphs $G$ and $H$, the generalized Ramsey number $R(G,H)$ is the smallest integer $N$ such that every $N$-vertex graph contains $G$ or its complement contains $H$. The numbers $R(G,H)$ have been extensively studied; see the survey~\cite{Radz}. We focus on finding the exact values of the numbers $R(K_{1,s},K_{m,n})$ --- the Ramsey numbers of stars versus complete bipartite graphs. Throughout the paper, we assume that $m \leq n$, using the symmetry of $K_{m,n}$. While the problem is simple for $s \leq 3$ or $m = 1$, it becomes more difficult for $s \geq 4$ and $m \geq 2$. For example, no general formula is known for the numbers $R(K_{1,s},K_{2,2})$, and the exact values are known only for some special cases~\cite{Par}--\cite{Boza}. 

In 1983, Burr et al.~\cite{BFRS} proved that, for fixed positive integers $s \geq 4$, \mbox{$m_1,\ldots,m_t \geq 2$} and all sufficiently large integers $n$,  

\begin{equation}\label{f1}
R(K_{1,s},K_{m_1,\ldots,m_t,n}) = n + (s-2) \sum_{i = 1}^t m_i + t + \epsilon,
\end{equation}
 where $\epsilon = 0$ if $s$ is even and $n + t$ is odd, and $\epsilon = 1$ otherwise. In 1990, Stevens~\cite{Stev} improved the sufficient threshold. In particular, the equality

\begin{equation}\label{f2}
R(K_{1,4},K_{m,n}) = n + 2m	 + 1 + (n \bmod 2)
\end{equation} 
holds for $n > m  = 2$ and for every fixed $m \geq 3$ and $ n \geq 3 \cdot 2^{m+1} - 2m - 3$. When $m$ and $n$ are sufficiently close, the formula~(\ref{f2}) gives only an upper bound. For example, we obtain $R(K_{1,4},K_{12,12}) = 33$, while the right-hand side of~(\ref{f2}) gives $37$. Apart from general results~\cite{BFRS,Stev}, there are two earlier papers on the numbers $R(K_{1,4},K_{m,n})$. Lortz and Mengersen proved~\cite{LoM4} that $R(K_{1,4},K_{3,3}) = R(K_{1,4},K_{3,4}) = 11$. Shao et al. proved~\cite{Shao} that \mbox{$R(K_{1,4},K_{3,5}) = R(K_{1,4},K_{4,4}) = 13$.}

In this paper, we determine all values $R(K_{1,4},K_{m,n})$ for a fixed small $m$. It turns out that this requires finding cubic graphs with given order and girth at least $m + 2$. Due to a recent result by Eze et al.~\cite{Eze}, this problem has been completely solved for all $m \leq 8$, while the case $m \geq 9$ remains open. For that reason, we obtain a complete solution only for $m \leq 8$. We also determine the diagonal numbers $R(K_{1,4},K_{m,m})$  for $9 \leq m \leq 12$.

\section{Terminology}

All graphs are finite and simple. A graph is called \textit{cubic} (\textit{subcubic}) if all its vertices have degree $3$ (at most $3$). Call a subcubic graph \textit{near-cubic} if all its vertices have degree $3$ except possibly one vertex of degree $2$. Call a vertex of a subcubic graph \textit{deficient} if it has degree less than $3$. A \textit{$k$-cycle} (\textit{$k$-path}) is a cycle (path) containing exactly $k$ vertices. Similarly, a \textit{$k^+$-cycle} and a \textit{$k^-$-cycle} are cycles containing at least $k$ vertices and at most $k$ vertices, respectively. An $xy$-path is a path with endpoints $x$ and $y$.

For a vertex $v \in V(G)$ and an integer $k \geq 0$, let $N_k(v)$ ($N_k[v]$) denote the set of vertices at distance exactly $k$ from $v$ (at most $k$ from $v$). Write $N(v)$ ($N[v]$) instead of $N_1(v)$ ($N_1[v]$). For a set $U \subseteq V(G)$, its \textit{closed $k$-neighborhood} is defined by $N_k[U] = \bigcup_{u \in U}N_k[u]$ and its \textit{$k$-th distance layer} is defined by $N_0(U) = U$ and $N_k(U) = N_k[U] \setminus N_{k-1}[U]$ for $k \geq 1$. Write $N(U)$ ($N[U]$) instead of $N_1(U)$ ($N_1[U]$). A vertex $v$ is called an \textit{external neighbor} of a set $U$, if $v \in N(U)$. For a cycle $C$, let $|C|$ denote the number of its vertices and $N_k[C]$ ($N_k(C)$) denote the set of vertices at distance at most~$k$ (exactly~$k$) from $C$.

For a graph $G$ and a set $U \subseteq V(G)$, denote by $G[U]$ the subgraph of $G$ induced by $U$. A vertex set $U' \subseteq V(G)$ is called \textit{connected} if the induced subgraph $G[U']$ is connected. A set $U \subseteq V(G)$ with $|U| = a$ and $|N(U)| = b$ is called an \textit{$(a,b)$-set}. We say that $U$ is an $(a,b^+)$-set ($(a,b^-)$-set) if $|U| = a$ and $|N(U)| \geq b$ ($|N(U)| \leq b$). A graph is called \textit{$(a,b)$-free} if it contains no $(a,b^-)$-set. For $k \ge 3$, call a vertex $v \in V(G)$ \textit{$k$-cyclic} if there exists a connected $k$-vertex set $U \subseteq V(G)$ such that $v \in U$ and $G[U]$ contains a cycle. Note that $v$ is not required to lie on this cycle. Moreover, if a vertex $v$ is $k$-cyclic and belongs to an $l$-vertex component of a graph, then $v$ is $k'$-cyclic for all $k \leq k' \leq l$.

Let $G$ be a graph with vertices $u$ and $v$, then $G + uv$ ($G - uv$) denotes a graph obtained from $G$ by adding (removing) the edge $uv$. For a subset $S \subseteq V(G)$, let $G - S$ denote the induced subgraph $G[V(G) \setminus S]$. If $S = \{v\}$, write $G - v$ instead of $G - \{v\}$. Subdividing an edge $uv$ of a graph $G$ means deleting it and adding a new degree-$2$ vertex $w$ adjacent to both $u$ and $v$. Suppressing a degree-$2$ vertex $v$ of a graph $G$ means deleting $v$ and adding an edge between its two neighbors. Note that if $G$ is triangle-free, then the resulting graph is simple.

For a graph $G$, the girth $g(G)$ is the number of vertices in a shortest cycle of $G$ (if $G$ is acyclic, then $g(G) = \infty$). A \textit{$(3,g)$-cage} is a smallest cubic graph with girth~$g$. A cubic graph with girth at least $g$ is called a \textit{$(3,g^+)$-graph}.

\section{Preliminaries}

\subsection{Girth conditions for near-cubic graphs}

While $(3,g)$-cages exist for all $g \geq 3$, their precise orders are currently known only for $g \leq 12$.  To date, all possible orders of $(3,g^+)$-graphs have been determined for $g \leq 10$~\cite{Eze}, and we use this fact in what follows.

\begin{lemma}{\cite[Section~5.1]{Eze}}\label{t}
Let $n$ be an even integer. The following hold:

\begin{enumerate}
\item An $n$-vertex $(3,5^+)$-graph exists if and only if $n \geq 10$.

\item An $n$-vertex $(3,6^+)$-graph exists if and only if $n \geq 14$.

\item An $n$-vertex $(3,7^+)$-graph exists if and only if $n \geq 24$. 

\item An $n$-vertex $(3,8^+)$-graph exists if and only if $n = 30$ or $n \geq 34$.

\item An $n$-vertex $(3,9^+)$-graph exists if and only if $n \geq 58$. 

\item An $n$-vertex $(3,10^+)$-graph exists if and only if $n \in \{70,72\}$ or $n \geq 76$.
\end{enumerate}
\end{lemma}
The next trivial observation is given without a proof.

\begin{lemma}\label{l1}
For a given $n \geq 10$ and $g \geq 5$, the following hold:

1. If there exists an $(n-1)$-vertex cubic graph $G$ with girth at least $g$, then any $n$-vertex graph obtained from $G$ by subdividing an edge is a near-cubic graph with girth at least $g$. Furthermore, its degree-$2$ vertex does not belong to a $g$-cycle.

2. If every cubic $(n-1)$-vertex graph has girth at most $g-1$, then a near-cubic $n$-vertex graph $G$ with a degree-$2$ vertex $v$ either has girth at most $g-1$ or it has girth $g$ and $v$ belongs to a $g$-cycle.
\end{lemma}

\subsection{Properties of $(a,b)$-sets}

The following simple observation relates the numbers $R(K_{1,4},K_{m,n})$ to $(a,b)$-free graphs.

\begin{observation}
The Ramsey number $R(K_{1,4},K_{m,n})$ is the smallest integer $N \geq m + n$ such that no $N$-vertex subcubic graph is $(m, N - m - n)$-free.
\end{observation}

\begin{lemma}\label{l2}
Let $G$ be a subcubic $n$-vertex graph and $U \subseteq V(G)$ be an $m$-vertex set containing $d_i$ vertices of degree $i$ in $G$, $i \in \{0,1,2\}$. Let $p_j$ be the number of vertices in $V(G) \setminus U$ with exactly $j \in \{2,3\}$ neighbors in $U$. Furthermore, let $c$ be the number of components in $G[U]$ and $s = |E(G[U])| - m + c$ be the cyclomatic number of $G[U]$. Then,  $U$ is an $(m,m + 2c - (3d_0 + 2d_1 + d_2 + p_2 + 2p_3 + 2s))$-set for all $1 \leq m \leq n$. 
\end{lemma}

\begin{proof}

The statement follows from the fact that
\begin{align*}|N(U)|  &= \sum_{u \in U}\deg(u) - 2|E(G[U])| - p_2 - 2p_3 \\ &= \bigl(3m - 3d_0 - 2d_1 - d_2\bigr) - 2 \cdot \bigl(m - c + s \bigr) - p_2 - 2p_3 \\ &= m + 2c - \bigl(3d_0 + 2d_1 + d_2 + p_2 + 2p_3 + 2s \bigr).
\end{align*}
\end{proof}

\begin{lemma}\label{l3}
Let $G$ be a subcubic $n$-vertex graph, and let $n \geq a \geq 3$ and $k \geq 1$. The following hold:

\begin{enumerate}
\item $G$ is not $(a,a + 2)$-free. Furthermore, any connected $k$-vertex set of $G$ is a $(k,(k+2)^-)$-set.

\item If $g(G) \leq a + 1$ or $\delta(G) \leq 2$, then $G$ is not $(a,a + 1)$-free. Furthermore, any connected $k$-vertex set of $G$ that contains all but one vertex of some cycle or a deficient vertex is a $(k,(k+1)^-)$-set.

\item Suppose that at least one of the following conditions holds:

\begin{enumerate}
\item $g(G) \leq a$;
\item $\delta(G) = 1$;
\item there are two vertices of degree $2$ at distance at most $a - 1$.
\end{enumerate}
Then $G$ is not $(a,a)$-free. Furthermore, any connected $k$-vertex set of $G$ that contains all vertices of some cycle, a degree-$1$ vertex or two degree-$2$ vertices is a $(k,k^-)$-set.
\end{enumerate}

\end{lemma}

\begin{proof}
We prove the third statement; the other two are proved similarly. Each of the conditions (a)--(c) yields the existence of a connected set $X \subseteq V(G)$ with at most $a$ vertices containing a cycle, a degree-$1$ vertex with its neighbor or two degree-$2$ vertices. If there exists a connected $a$-vertex set $U$ containing $X$, then $U$ is an $(a,a^-)$-set by Lemma~\ref{l2}. Otherwise, let $U$ consist of the whole component containing $X$, possibly some other whole components, and a (possibly empty) connected set $Y$ in one further component. Observe that $|X| \geq 2$, thus $|Y| \leq a - 2$, and $|N(U)| = |N(Y)| \leq a$ by Lemma~\ref{l2}. Hence $G$ is not $(a,a)$-free as required. The last assertion follows immediately from Lemma~\ref{l2}.
\end{proof}

\begin{lemma}\label{l4}
Let $G$ be a cubic graph with $g(G) \geq 5$. Then $G$ is $(a,a + 1)$-free for every $3 \leq a \leq g(G) - 2$.
\end{lemma}

\begin{proof}
Suppose, to the contrary, that $G$ has an $(a,(a+1)^-)$-set $U$. Since \mbox{$g(G) > a$}, the subgraph $G[U]$ is acyclic. Let $\overline{U} = V(G) \setminus U$ and let $c$ be the number of components in $G[U]$.

\textit{Case 1.} $c \leq 2$. If $c = 1$, no two vertices of $G[U]$ have a common neighbor outside $U$, hence $U$ is an $(a,a+2)$-set by Lemma~\ref{l2} (take $c = 1$ and $d_i = p_j = s = 0$). Suppose that $G[U]$ has two components, $U_1$ and $U_2$. Since $g(G) > a$, no vertex of $\overline{U}$ has two distinct neighbors in the same component of $G[U]$. If, moreover, $\overline{U}$ has at most two vertices with two neighbors in $U$, then $U$ is an $(a,(a+2)^+)$-set by Lemma~\ref{l2}. Otherwise, consider three vertices $v_1,v_2,v_3 \in \overline{U}$ with neighbors in both $U_1$ and $U_2$. Let~$P_1$, respectively $P_2$, be a $v_1v_2$-path whose internal vertices lie in $U_1$, respectively in~$U_2$. Then the cycle $C = v_1P_1v_2P_2v_1$ has at most $a+2 \leq g(G)$ vertices, with equality only if $C$ contains all vertices of $U$. The vertex $v_3$ has two neighbors on a $5^+$-cycle $C$, hence there is a  shorter cycle passing through $v_3$, a contradiction with $|C| = g(G) =  a + 2$.

\textit{Case 2.} $c \geq 3$. The subgraph $H = G[N[U]]$ has at most $2a + 1$ vertices. Since $g(H) \geq g(G) \geq a + 2$, any two cycles of $H$ have a common vertex. Since $G$ is cubic and $H[U]$ is acyclic with $a - c$ edges, there are $a + 2c$ edges between $U$ and $N_1(U)$, hence $H$ has at least $2a + c$ total edges and exactly one component $H'$ with cycles. Furthermore, $|E(H')| \geq |V(H')| + 2$. Consider an inclusion-wise minimal subgraph $H''$ of $H'$ with $|E(H'')| = |V(H'')| + 2$. Observe that $H''$ has minimum degree at least~$2$ (otherwise it is possible to remove one vertex and one edge and obtain a smaller subgraph). Moreover, every cycle of $H''$ is a cycle of $H$, hence $H''$ does not contain disjoint cycles. Thus $H''$ is connected with exactly four degree-3 vertices. If every pair of these vertices is joined by a path whose internal vertices have degree $2$ in $H''$, then $H''$ is a subdivision of $K_4$. Consider the four cycles of $H''$ that contain exactly three degree-$3$ vertices. These cycles have at least $4a+8$ total edges. However, every edge of $H''$ belongs to exactly two of them and there are at most $2a+3$ edges, a contradiction. If there exist two degree-$3$ vertices of $H''$ not joined by such a path, then $H''$ contains a cycle $C$ that passes through at most two degree-$3$ vertices of~$H''$. Then $H'' - V(C)$ has another cycle and $H''$ contains two disjoint cycles, a contradiction.
\end{proof}

\begin{lemma}\label{cyclic}
Let $G$ be a connected $n$-vertex near-cubic graph with a degree-$2$ vertex $v$. For all $1 \leq k \leq n/2 - 1$, the following hold:

1. If $v$ is not $(2k+1)$-cyclic, then the subgraph $G[N_k[v]]$ is a $(2^{k+1}-1)$-vertex tree with exactly $2^s$ vertices in $N_s(v)$ for all $1 \leq s \leq k$.

2. If $v$ is not $(2k+2)$-cyclic, then the subgraph $G[N_k[v]]$ is a $(2^{k+1} - 1)$-vertex tree with exactly $2^s$ vertices in $N_s(v)$ for all $1 \leq s \leq k$. Furthermore, $N_{k+1}(v)$ has exactly $2^{k + 1}$ vertices.
\end{lemma}

\begin{proof}
We prove the second statement; it is straightforward to prove the first statement analogously. If, for all $1 \leq s \leq k$, every vertex in $N_s(v)$ has exactly two neighbors in $N_{s+1}(v)$ that are not adjacent to other vertices of $N_s(v)$, then it is easy to see that the statement holds. Otherwise, let $r \leq k$ be the smallest integer such that two vertices $w_1,w_2 \in N_r(v)$ are adjacent or have a common neighbor in $N_{r+1}(v)$. If $w_1w_2 \in E(G)$, take a shortest $vw_1$-path and a shortest $vw_2$-path; together, these paths form a connected set with a cycle on at most $2r+1$ vertices, hence $v$ is $(2k+2)$-cyclic, a contradiction. If $w_1w_2 \notin E(G)$, let $w_3 \in N_{r+1}(v)$ be their common neighbor. Consider two distinct $(r+2)$-vertex $vw_3$-paths, one passing through $w_1$, and another passing through $w_2$. Together, these paths form a connected set with a cycle on at most $2r+2$ vertices, hence $v$ is $(2k+2)$-cyclic, again a contradiction.
\end{proof}

\subsection{Edge-maximality and greedy component selection}

\begin{lemma}\label{edge}
Let $G$ be a subcubic $(a,b)$-free graph, where $a,b \geq 3$. Then there exists a subcubic $(a,b)$-free graph $G'$ that contains $G$ as a spanning subgraph and whose deficient vertices form a (possibly empty) clique.
\end{lemma}

\begin{proof}
If $G$ has a pair of nonadjacent deficient vertices $u$ and $v$, add an edge~$uv$. The resulting graph $G + uv$ is subcubic $(a,b)$-free, because if $U \subseteq V(G)$ is an $(a,(b+1)^+)$-set of $G$, then it is an $(a,(b+1)^+)$-set of $G + uv$. Continue to add edges while possible, until we obtain a graph $G'$ satisfying the condition of the lemma.
\end{proof}

Let $G$ be a disconnected subcubic graph all but at most one of whose components are cubic, and let $6 \leq k \leq |V(G)|$. Consider the following procedure called \textit{greedy component selection} that constructs a $k$-vertex subset of $V(G)$. Start with an empty subset and stop as soon as $k$ vertices have been added. If the number of vertices of some component does not exceed the number of vertices that remain to be added, add this component entirely. If this is no longer possible and $r \geq 1$ vertices remain, fix a smallest component $G'$ that has not yet been chosen. If $G'$ is acyclic, select any connected $r$-vertex subset of $G'$. Otherwise, select a connected $r$-vertex subset that either contains at least one cycle of $G'$ or all its vertices belong to a shortest cycle of $G'$. Call a $k$-vertex subset $U \subseteq V(G)$ \textit{greedy} if it can be obtained by greedy component selection (generally speaking, this set is not uniquely determined). The next lemma shows that greedy component selection is useful to construct sets with small neighborhoods in disconnected graphs.

\begin{lemma}\label{greedy}
Let $a \geq 6$ and $G$ be a disconnected subcubic graph on at most $4a - 1$ vertices all but at most one component of which are cubic with girth at most $a-1$. Then every greedy $a$-vertex subset of $G$ is an~$(a,(a-1)^-)$-set.
\end{lemma}

\begin{proof}
 Let $X$ be any greedy $a$-vertex subset of $G$ and $Y$ be a connected $r$-vertex part of $X$ that was selected last. Then $|N(X)| = |N(Y)|$. If $N(Y) = \emptyset$, there is nothing to prove; otherwise, we show that $|N(Y)| \leq a - 1$. If $r \leq a - 3$, this follows from Lemma~\ref{l3}(1). If $r = a - 2$ ($r = a - 1$), then $Y$ lies in a cubic component of $G$, because the vertices chosen before form a whole noncubic component. Thus~$Y$ contains all but at most one (all) vertices of some cycle and $|N(Y)| \leq a - 1$ by Lemma~\ref{l3}(2) (by Lemma~\ref{l3}(3)). Finally, if $r= a$, then  all vertices of $X$ belong to a smallest component of $G$ on at most $2a-1$ vertices and thus $|N(Y)| \leq a - 1$.
\end{proof}

\subsection{The upper bound}

We now show that the right-hand side of~(\ref{f2}) is an upper bound for all $3 \leq m \leq n$, and characterize when the bound $2m + n + 2$ is attained.

\begin{lemma}\label{l6}
For all $3 \leq m \leq n$, we have $R(K_{1,4},K_{m,n}) \leq 2m + n + 2$. Equality holds if and only if there exists a cubic $(2m + n + 1)$-vertex graph with girth at least~$m + 2$.
\end{lemma}

\begin{proof}
Suppose first that there exists a cubic $(2m + n + 1)$-vertex graph $G$ with girth at least $m + 2$. By Lemma~\ref{l4}, $G$ is $(m,m + 1)$-free. Thus $R(K_{1,4},K_{m,n}) > 2m + n + 1.$ On the other hand, by Lemma~\ref{l3}(1), there is no $(m,m + 2)$-free subcubic graph on $2m + n + 2$ vertices. Therefore, $R(K_{1,4},K_{m,n}) \leq 2m + n + 2$ and the equality holds.

Conversely, suppose that no cubic graph on $2m + n + 1$ vertices has girth at least $m+2$. Lemma~\ref{l3}(2) implies that every subcubic $(2m + n + 1)$-vertex graph $H$ is not $(m,m + 1)$-free, thus $R(K_{1,4},K_{m,n}) \leq 2m + n + 1.$

\end{proof}

\section{Main results}

\subsection{The lower bounds}

In this subsection, we obtain the constructions for the lower bounds. Most of them follow from Lemmas~\ref{lower1} and~\ref{lower2}; the others were found by computer search and are given in \verb|sparse6| format in Table~\ref{table1}.

\begin{lemma}\label{lower1}
For every $a \geq 5$, a near-cubic graph is $(a,a)$-free if it has girth at least $a+1$ and its degree-$2$ vertex, if any, does not belong to an $(a+1)$-cycle.
\end{lemma}

\begin{proof}
Suppose that a graph $G$ satisfying the hypotheses has an $(a,a^-)$-set $U$. First, suppose that $G$ is cubic. Lemma~\ref{l4} implies that, for every vertex $u \in U$, $U \setminus \{u\}$ is an $(a-1,(a+1)^+)$-set, thus $|N(U \setminus \{u\})| > |N(U)|$. This is possible only if the following three properties hold: $|N(U)| = a$, every vertex of $U$ has a neighbor in $U$ and every vertex of $N(U)$ has at least two neighbors in $U$ (if, for example, a vertex $x \in N(U)$ has only one neighbor, $u$, in $U$, then $|N(U \setminus \{u\})| \leq |N(U)|$, a contradiction). Consider the bipartite subgraph $B$ of $G$ with partite sets $U$ and $N(U)$ and all edges from $U$ to $N(U)$ that $G$ has. Every vertex from $N(U)$ has at least two neighbors in $U$ and every vertex of $U$ has at most two neighbors in $N(U)$; hence the equality $|U| = |N(U)|$ implies that $B$ is $2$-regular and $G[U]$ is $1$-regular. Furthermore, $g(B) \geq g(G) > a$, hence $B$ is a Hamiltonian $2a$-cycle. Its spanning supergraph $G[N[U]]$ contains $a/2 \geq 3$ additional edges from $G[U]$, thus it has an $a^-$-cycle; a contradiction with $g(G) \geq a + 1$.

Suppose that $G$ is near-cubic and $v$ is its only degree-$2$ vertex with neighbors $w_1$ and $w_2$. Consider a cubic graph $G'$ obtained from $G$ by suppressing $v$. Then $g(G') \geq a + 1$ and the argument above implies that $G'$ is $(a,a)$-free. If $v \in U$, Lemma~\ref{l4} implies that $U' = U \setminus \{v\}$ is an $(a-1,(a+1)^+)$-set in $G'$. Furthermore, every external neighbor of $U'$ in $G'$ is an external neighbor of $U$ in $G$, thus $U$ is an $(a,(a+1)^+)$-set in $G$, a contradiction. If $v \notin U$, then $U$ is an $(a,(a+1)^+)$-set in $G'$, hence it is also an $(a,(a+1)^+)$-set in $G$, again a contradiction.
\end{proof}

\begin{lemma}\label{lower2}
For every $a \geq 7$, a near-cubic graph is $(a,a-1)$-free if it has girth at least $a$ and its degree-$2$ vertex, if any, does not belong to an $a$-cycle.
\end{lemma}

\begin{proof}
Suppose that a graph $G$ satisfying the hypotheses has an $(a,(a-1)^-)$-set $U$. First, suppose that $G$ is cubic. For every vertex $u \in U$, we have $|N(U \setminus \{u\})| > |N(U)|$ by the previous lemma. Again, this implies that $|N(U)| = a - 1$, every vertex of $U$ has a neighbor in $U$ and every vertex of $N(U)$ has at least two neighbors in $U$. Therefore, the subgraph $G[N[U]]$ is subcubic with minimum vertex degree at least $2$. Every vertex from $U$ has degree $3$ in $G[N[U]]$, thus $G[N[U]]$ has at least $8$ degree-$3$ vertices. Let $H$ be a $2k$-vertex cubic graph ($k \geq 4$) obtained from $G[N[U]]$ by suppressing all its degree-$2$ vertices. Note that $H$ is simple, because $g(G) \geq 7$, all suppressed vertices have both neighbors in $U$ and, furthermore, two degree-$2$ vertices cannot have two common neighbors, otherwise they would form a $4$-cycle in $G$. We prove that $H$ contains two vertex-disjoint cycles. Let $C$ be a shortest cycle of $H$. If $|C| \geq k + 1$, then there are two vertices of $C$ with a common neighbor outside $C$, which is impossible for $k \geq 4$. Thus $|C| = k - r$ for some $0 \leq r \leq k - 3$ and there are $2(k-r)$ edges incident to vertices of $C$. Hence a $(k+r)$-vertex subgraph $H[V(H) \setminus V(C)]$ has at least $k + 2r$ edges, so it has a cycle $C'$ such that $C$ and $C'$ are disjoint. These cycles correspond to disjoint cycles in $G[N[U]]$, which is a contradiction, since $g(G[N[U]]) \geq a$ and $|N[U]| < 2a$.

Suppose that $G$ is near-cubic and $v$ is its only degree-$2$ vertex with neighbors $w_1$ and $w_2$. Consider a cubic graph $G'$ obtained from $G$ by suppressing $v$. Then $g(G') \geq a$ and the argument above implies that $G'$ is $(a,a-1)$-free. If $v \in U$, Lemma~\ref{lower1} implies that $U' = U \setminus \{v\}$ is an $(a-1,a^+)$-set in $G'$. Furthermore, every external neighbor of $U'$ in $G'$ is an external neighbor of $U$ in $G$, thus $U$ is an $(a,a^+)$-set in $G$, a contradiction. If $v \notin U$, then $U$ is an $(a,a^+)$-set in $G'$, hence also an $(a,a^+)$-set in $G$, again a contradiction.
\end{proof}

\begin{lemma}\label{m4}
There exists a $12$-vertex cubic $(4,4)$-free graph.
\end{lemma}

\begin{proof}
Consider the graph $H_{12}$; its \verb|sparse6| representation is given in Table~\ref{table1}. Direct computation shows that every 4-vertex subset of $H_{12}$ has at least five external neighbors, hence $H_{12}$ is $(4,4)$-free.
\end{proof}

\begin{lemma}\label{m5}
For all $n \geq 5$, there exists an $(n + 10)$-vertex subcubic $(5,5)$-free graph.
\end{lemma}

\begin{proof}
Lemma~\ref{t}(2) and Lemma~\ref{l1}(1) imply that for every $n \geq 5$ there is an $(n+10)$-vertex near-cubic graph of girth at least $6$ such that its degree-$2$ vertex, if any, does not belong to a $6$-cycle. Such a graph is $(5,5)$-free by Lemma~\ref{lower1}.
\end{proof}

\begin{lemma}\label{m6}
The following hold:

1. For all $6 \leq n \leq 11$, there exists an $(n + 11)$-vertex subcubic $(6,5)$-free graph.

2. For all $n \geq 12$, there exists an $(n + 12)$-vertex subcubic $(6,6)$-free graph.
\end{lemma}

\begin{proof}
1. The graphs $H_{17},H_{18},\ldots,H_{22}$ (see Table~\ref{table1}) are $(6,5)$-free.

2. Apply Lemma~\ref{t}(3), Lemma~\ref{l1}(1) and Lemma~\ref{lower1}.
\end{proof}

\begin{table}[h!]
\small
\centering
\begin{tabular}{|c|l|}
\hline
Graph & Representation in \verb|sparse6| \\
\hline
$H_{12}$  &  \verb|:K`ESyTl?{@eaO{ag| \\ \hline
$H_{17}$  &  \verb|:Pa@abcd_E_`abGcHdGIeHJfIK_@| \\ \hline
$H_{18}$ & \verb|:Q_`ab_Cde`FgbHeIjaKfLcMhN_KO|   \\ \hline
$H_{19}$ & \verb|:Ra@ab_Cde`FgbHeIjaKfLcMhN_KO_@|  \\ \hline
$H_{20}$ & \verb|:S___``aabbcEcGdFdeLfGhIhKjMPlNO| \\ \hline
$H_{21}$ & \verb|:Ta?_``aabbcEcGdFdeLfGhIhKjMPlNO_@| \\ \hline
$H_{22}$ & \verb|:U_`abcdefg_H_`abJcKdLeMfNgJOhKPiLQ| \\ \hline
$H_{23}$ & \verb|:V_b?``aabbcEcGdEdHfGfHiMjkNPlO_ArSpT| \\ \hline
$H_{24}$ & \verb|:W_`abcde`FghcI_JkfLbMniOePaQlRhSdT_OU| \\ \hline
$H_{25}$  &  \verb|:Xa?_fEegfkCHdIcJdJcdO`EJbGPaFObIMaHN`KL_@| \\ \hline
$H_{26}$ & \verb|:Y_`abcdefghi_J_`abcdLeMfNgLOhMPiNQjORkPS| \\ \hline
$H_{28}$ & \verb|:[___fCdechIdeKfgMgfchOePdQ`IKaLMbJN`OPbSTaQR|   \\ \hline
$H_{30}$ & \verb|:]___``aabbccddeeffgghhiMQiOSjNRjPTkMTkORlNSlPQ|  \\ \hline
$H_{32}$ & \verb|:__`abcdefghijkl_M_`abcdeOfPgQhRiOSjPTkQUlRVmSWnTX| \\ \hline
$H_{33}$ & \verb|:`_OGKEDA`owcQJDbPg{]PGcqXSiUOHqx_{hKGdXhYYPHra@U[QIRqpacWL~| \\ \hline
\end{tabular}
\caption{Witness graphs for Lemma~\ref{m4} and Lemmas~\ref{m6}--\ref{m912}}
\label{table1}
\end{table}

\begin{lemma}\label{m7}

The following hold:

\begin{enumerate}
\item There exists a $20$-vertex subcubic $(7,5)$-free graph.

\item For $n = 7$ and all $n \geq 9$, there exists a subcubic $(n+13)$-vertex $(7,6)$-free graph.

\item For $n \in \{16,17\}$ and $n \geq 20$, there exists a subcubic $(n+14)$-vertex $(7,7)$-free graph.
\end{enumerate}

\end{lemma}

\begin{proof}
1. The graph $H_{20}$ (see Table~\ref{table1}) is $(7,5)$-free.

2. The graphs  $H_{20},H_{22},H_{23}$ (see Table~\ref{table1}) are $(7,6)$-free. For $n \geq 11$, apply Lemma~\ref{t}(3), Lemma~\ref{l1}(1) and Lemma~\ref{lower2}.

3. Apply Lemma~\ref{t}(4), Lemma~\ref{l1}(1) and Lemma~\ref{lower1}.
\end{proof}

\begin{lemma}\label{m8}

The following hold:

\begin{enumerate}
\item For all $n \in \{8, 10, 12, 14\}$, there exists a subcubic $(n + 14)$-vertex $(8,6)$-free graph.

\item For all $n \in \{9,11,13\}$ and $n \geq 15$, there exists a subcubic $(n+15)$-vertex $(8,7)$-free graph.

\item For $n \geq 42$, there exists a subcubic $(n+16)$-vertex $(8,8)$-free graph.
\end{enumerate}

\end{lemma}

\begin{proof}

1. Direct computation shows that the graphs $H_{22}$, $H_{24}, H_{26},H_{28}$ are $(8,6)$-free.

2. Direct computation shows that the graphs $H_{24}, H_{26}, H_{28},H_{32},H_{33}$ are $(8,7)$-free. For $n \in \{15,16\}$ and all $n \geq 19$, apply Lemma~\ref{t}(4), Lemma~\ref{l1}(1) and Lemma~\ref{lower2}.

3. Apply Lemma~\ref{t}(5), Lemma~\ref{l1}(1) and Lemma~\ref{lower1}.
\end{proof}

\begin{lemma}\label{m912}
The following hold:

1. There exists a $25$-vertex subcubic $(9,7)$-free graph.

2. There exists a $28$-vertex subcubic $(10,8)$-free graph.

3. There exists a $30$-vertex subcubic $(11,8)$-free graph.

4. There exists a $32$-vertex subcubic $(12,8)$-free graph.
\end{lemma}

\begin{proof}
Direct computation shows that the graphs $H_{25},H_{28},H_{30},H_{32}$ (see Table~\ref{table1}) satisfy the conditions $(1)$--$(4)$, respectively.
\end{proof}

\subsection{Case $m \leq 8$}

The following trivial fact was first mentioned in~\cite{Har1}.

\begin{theorem}\label{t1}
For every $n \geq 1$, $R(K_{1,4},K_{1,n}) = n + 3 + (n \bmod 2)$.
\end{theorem}

The next fact was proven in~\cite{Par} for $n = 2$. For $n \geq 3$, it follows from a general result by Stevens~\cite{Stev}.

\begin{theorem}\label{t2}
For every $n \geq 2$, $R(K_{1,4},K_{2,n}) = n + 5 + (n \bmod 2)$.
\end{theorem}

We start with $m = 3$, which is the first previously open case.

\begin{theorem}\label{t3}
For every $n \geq 3$, $R(K_{1,4},K_{3,n}) = n + 7 + (n \bmod 2)$.
\end{theorem}

\begin{proof}
Since every cubic graph has even order, Lemma~\ref{t}(1) implies that, for every $n \geq 3$, a cubic graph on $n + 7$ vertices with girth at least $5$ exists if and only if $n$ is odd. Hence, by Lemma~\ref{l6}, $R(K_{1,4},K_{3,n}) = n + 8$
for every odd $n \geq 3$, while 
$R(K_{1,4},K_{3,n}) \leq n + 7$ for every even $n \geq 4$. It remains to prove that, for every even $n \geq 4$, there exists an $(n+6)$-vertex subcubic $(3,3)$-free graph. Indeed, Lemma~\ref{t}(1) implies that for every even $N \geq 10$, there exists a cubic $N$-vertex graph with girth at least 5. Lemma~\ref{l4} implies that every such graph is $(3,4)$-free, and hence $(3,3)$-free.
\end{proof}

\begin{theorem}\label{t4}
For every $n \geq 4$, $R(K_{1,4},K_{4,n}) = n + 9 + (n \bmod 2)$.
\end{theorem}

\begin{proof}
For $n \geq 5$, the proof is analogous to that of Theorem~\ref{t3}. For $n = 4$, the upper bound follows from Lemma~\ref{l6}, while the lower bound follows from Lemma~\ref{m4}.
\end{proof}

\begin{theorem}\label{t5}
The following hold:

\begin{enumerate}
\item $R(K_{1,4},K_{5,n}) = n + 11$ for $5 \leq n \leq 12$;

\item $R(K_{1,4},K_{5,n}) = n + 11 + (n \bmod 2)$ for $n \geq 13$.
\end{enumerate}
\end{theorem}

\begin{proof}
We prove the first statement; the second statement is proved analogously to Theorem~\ref{t3}. The lower bound follows from Lemma~\ref{m5}. By Lemma~\ref{t}(3), for every $n \leq 12$, an $(n+11)$-vertex cubic graph has girth at most~6. Hence the upper bound follows from Lemma~\ref{l6}.
\end{proof}

\begin{theorem}\label{t6}
The following hold:

\begin{enumerate}
\item $R(K_{1,4},K_{6,n}) = n + 12$ for $6 \leq n \leq 11$;

\item $R(K_{1,4},K_{6,n}) = n + 13$ for $12 \leq n \leq 16$ and for $n \in \{19,20\}$;

\item $R(K_{1,4},K_{6,n}) = n + 13 + (n \bmod 2)$ for $n \in \{17,18\}$ and for all $n \geq 21$.
\end{enumerate}
\end{theorem}

\begin{proof}
We prove the first statement. The lower bound follows from Lemma~\ref{m6}(1). We prove the upper bound. Suppose there exists an $(n+12)$-vertex subcubic $(6,6)$-free graph $G$. By Lemma~\ref{edge}, we may assume that the deficient vertices of $G$ form a clique. If $G$ is disconnected, then each of its cubic components has girth at most $6$ by Lemma~\ref{t}(3), thus $G$ is not $(6,6)$-free by Lemma~\ref{l3}(3). If $G$ is connected, but not near-cubic, then it contains either a degree-$1$ vertex or two adjacent degree-$2$ vertices, hence any connected $6$-vertex set containing all deficient vertices of $G$ is a $(6,6^-)$-set by Lemma~\ref{l2}. Finally, if $G$ is connected and near-cubic, then Lemma~\ref{t}(3) and Lemma~\ref{l1}(2) imply that it either contains a $6^-$-cycle or its degree-$2$ vertex belongs to a $7$-cycle, hence $G$ is not $(6,6)$-free, again a contradiction.

We now prove the second statement. The lower bound follows from Lemma~\ref{m6}(2). By Lemma~\ref{t}(4), for $12 \leq n \leq 16$ and $n \in \{19,20\}$, a cubic graph on $(n + 13)$ vertices has girth at most 7, hence the upper bound follows from Lemma~\ref{l6}.

The third statement is proved analogously to Theorem~\ref{t3}.
\end{proof}

\begin{theorem}\label{t7}
The following hold:

\begin{enumerate}
\item $R(K_{1,4},K_{7,n}) = n + 13$ for $n = 8$;

\item $R(K_{1,4},K_{7,n}) = n + 14$ for $9 \leq n \leq 15$ and for $n \in \{7, 18, 19\}$;
 
\item $R(K_{1,4},K_{7,n}) = n + 15$ for $n \in \{16,17\}$ and for $20 \leq n \leq 42$;

\item $R(K_{1,4},K_{7,n}) = n + 15 + (n \bmod 2)$ for $n \geq 43$.
\end{enumerate}
\end{theorem}

\begin{proof}
We prove the first statement.  The lower bound follows from Lemma~\ref{m7}(1). To prove the upper bound, we show that there is no $21$-vertex $(7,6)$-free subcubic graph. Suppose, to the contrary, that such a graph $G$ exists. By Lemma~\ref{edge}, we may assume that the deficient vertices of $G$ form a clique. If $G$ is disconnected, then each of its cubic components has girth at most $6$ by Lemma~\ref{t}(3), hence $G$ is not $(7,6)$-free by Lemma~\ref{greedy}. If $G$ is connected, but not near-cubic, then, since $|V(G)|$ is odd, it contains a degree-$1$ vertex $x$ adjacent to a degree-$2$ vertex $y$, and any $7$-vertex connected set containing $x$ and $y$ is a $(7,6^-)$-set by Lemma~\ref{l2}, a contradiction. If $G$ is connected and near-cubic, consider its only degree-$2$ vertex $v$. If $v$ is $7$-cyclic, there is a connected $(7,6^-)$-set containing $v$ together with some cycle. Otherwise Lemma~\ref{cyclic} implies that the subgraph $G[N_3[v]]$ is a $15$-vertex tree and $N_3(v)$ is an independent 8-vertex set. Furthermore, there are $16$ edges between $N_3(v)$ and $N_4(v)$. Since $|N_4(v)| \leq |V(G) \setminus N_3[v]| \leq  6$, there exists a vertex $x \in N_4(v)$ with three neighbors in $N_3(v)$. Then there exist vertices $y,z \in N_3(v) \cap N_1(x)$ such that the $4$-vertex $vy$-path and the $4$-vertex $vz$-path have at least two common vertices. Consider a connected $7$-vertex set $U$ that contains these paths together with $x$. Since $U$ has a degree-$2$ vertex and a cycle, it is a $(7,6^-)$-set by Lemma~\ref{l2}, a contradiction.

We prove the second statement.  The lower bound follows from Lemma~\ref{m7}(2). For the upper bound, suppose that there exists a subcubic $(n+14)$-vertex $(7,7)$-free graph $G$. By Lemma~\ref{edge}, we may assume that the deficient vertices of $G$ form a clique. Suppose that $G$ is disconnected. If it contains a cubic component of girth at most $7$, then it is not $(7,7)$-free by Lemma~\ref{l3}(3). Otherwise, $G$ has two components: a $30$-vertex component of girth $8$ and an $r$-vertex clique, where $r \in \{2,3\}$. Thus $G$ either contains a degree-$1$ vertex or a $3$-cycle, hence it is not $(7,7)$-free by Lemma~\ref{l3}(3). Suppose that $G$ is connected, but not near-cubic. Then any connected $7$-vertex set that contains all its deficient vertices is a $(7,7^-)$-set. Finally, if $G$ is connected and near-cubic, Lemma~\ref{t}(4) and Lemma~\ref{l1}(2) imply that either $g(G) \leq 7$ or $g(G) = 8$ and the only degree-$2$ vertex of $G$ belongs to an $8$-cycle. Hence $G$ is not $(7,7)$-free, a contradiction.

We prove the third statement. The lower bound follows from Lemma~\ref{m7}(3). By Lemma~\ref{t}(5), for $n \leq 42$, there is no cubic $(n+15)$-vertex graph with girth at least $9$, hence the upper bound follows from Lemma~\ref{l6}.

The last statement is proved analogously to Theorem~\ref{t3}.

\end{proof}

\begin{theorem}\label{t8}
The following hold:

\begin{enumerate}
\item $R(K_{1,4},K_{8,n}) = n + 15$ for $n \in \{8, 10, 12, 14\}$;

\item $R(K_{1,4},K_{8,n}) = n + 16$ for $n \in \{9,11,13\}$ and $15 \leq n \leq 41$;
 
\item $R(K_{1,4},K_{8,n}) = n + 17$ for $42 \leq n \leq 52$ and $n \in \{57,58\}$;

\item $R(K_{1,4},K_{8,n}) = n + 17 + (n \bmod 2)$ for $n \in \{53,54,55,56\}$ and $n \geq 59$.
\end{enumerate}
\end{theorem}

\begin{proof}
We prove the first statement. The lower bound follows from Lemma~\ref{m8}(1). We prove the upper bound. Suppose there exists a subcubic $(n+15)$-vertex $(8,7)$-free graph~$G$. By Lemma~\ref{edge}, we may assume that the deficient vertices of $G$ form a clique. If $G$ is disconnected, then each of its cubic components has girth at most~$7$ by Lemma~\ref{t}(4), hence $G$ is not $(8,7)$-free by Lemma~\ref{greedy}. If $G$ is connected, but not near-cubic, then, since $|V(G)|$ is odd, there exists a degree-1 vertex adjacent to a degree-2 vertex, and any connected $8$-vertex set containing them is an $(8,7^-)$-set. Finally, suppose that $G$ is near-cubic and let $v$ be its only degree-$2$ vertex. Lemma~\ref{cyclic} implies that $v$ is $8$-cyclic, hence there exists a connected $(8,7^-)$-set containing it, a contradiction.

We prove the second statement. The lower bound follows from Lemma~\ref{m8}(2). We now prove the upper bound. Suppose there exists a subcubic $(n+16)$-vertex $(8,8)$-free graph~$G$. By Lemma~\ref{edge}, we may assume that the deficient vertices of $G$ form a clique. If $G$ is disconnected, then each of its cubic components has girth at most $8$ by Lemma~\ref{t}(5), thus it is not $(8,8)$-free by Lemma~\ref{l3}(3); if $G$ is connected and cubic, we obtain the same contradiction. If $G$ is connected, but not near-cubic, then any connected $8$-vertex set that contains all its deficient vertices is an $(8,8^-)$-set. Finally, suppose that $G$ is near-cubic with a degree-$2$ vertex. By Lemma~\ref{t}(5) and Lemma~\ref{l1}(2), either $G$ has girth at most $8$, or its degree-$2$ vertex belongs to a $9$-cycle. Hence $G$ is not $(8,8)$-free, a contradiction.

We prove the third statement. The lower bound follows from Lemma~\ref{m8}(3). By Lemma~\ref{t}(6), for $42 \leq n \leq 52$ and $n \in \{57,58\}$, there is no $(n+17)$-vertex cubic graph with girth at least $10$, hence the upper bound follows from Lemma~\ref{l6}.

The last statement is proved analogously to Theorem~\ref{t3}.
\end{proof}

\subsection{Case $9 \leq m \leq 12$}

\begin{theorem}\label{t9}
$R(K_{1,4},K_{9,9}) = 26$.
\end{theorem}

\begin{proof}
The lower bound follows from Lemma~\ref{m912}(1); we now prove the upper bound. Suppose, to the contrary, that there exists a  $26$-vertex subcubic $(9,8)$-free graph $G$. By Lemma~\ref{edge}, we may assume that the deficient vertices of $G$ form a clique. If, moreover, $G$ is disconnected, then each of its cubic components has girth at most~7 by Lemma~\ref{t}(4), thus it is not $(9,8)$-free by Lemma~\ref{greedy}. We now suppose that $G$ is connected.

\textit{Case~$1.$} $\delta(G) = 1$. Let $v$ be a degree-$1$ vertex of $G$ and $w$ be its only neighbor. Since $|V(G)|$ is even, all other vertices of $G$ have degree 3 by Lemma~\ref{edge}. Lemma~\ref{cyclic} implies that $w$ is $8$-cyclic in $G - v$, thus it is also $8$-cyclic in $G$. Hence there is a connected $9$-vertex set containing $v$ and some cycle, which is a $(9,7^-)$-set by Lemma~\ref{l2}, a contradiction. 

\textit{Case~$2.$} $\delta(G) = 2$. Since $|V(G)|$ is even, $G$ contains adjacent degree-$2$ vertices $u$ and $v$ by Lemma~\ref{edge}. If the subgraph $G[N_3[\{u,v\}]]$ is not acyclic, then it is easy to see that both $u$ and $v$ are $8$-cyclic, and there is a connected $9$-vertex set that contains both of them and a cycle, which is a $(9,7^-)$-set, a contradiction. If $G[N_3[\{u,v\}]]$ is acyclic, then $|N_3[\{u,v\}]| = 16$ and $|N_4(\{u,v\})| \leq 10$, hence there exist two vertices from $N_3(\{u,v\})$ with a common neighbor in $N_4(\{u,v\})$. Then both $u$ and $v$ are $9$-cyclic and we obtain the same contradiction.

\textit{Case~$3.$} $\delta(G) = 3$. Since $G$ is cubic, Lemma~\ref{t}(4) implies that $g(G) \leq 7$. Let $C$ be a shortest cycle in $G$.

\textit{Case~$3.1.$ $|C| = 3$.} Suppose that the subgraph $G[N_3[C]]$ is unicyclic. Then $|N_3[C]| = 24$, $|N_4(C)| \leq 2$ and, furthermore, $N_3(C)$ is a $12$-vertex independent set, which is impossible, since there are at most $6$ edges from $N_3(C)$ to $N_4(C)$. Hence $G[N_3[C]]$ is not unicyclic and there exists a connected $9$-vertex set containing $C$ and some other cycle, which is a $(9,7^-)$-set by Lemma~\ref{l2}, a contradiction.

\textit{Case~$3.2.$ $|C| \in \{4,5\}$.} Note that $|V(G)| < 8 \cdot |C|$. If $G[N_2[C]]$ is unicyclic, then $|N_1(C)| = |C|$, $|N_2(C)| = 2 \cdot |C|$, and thus $|N_3(C)| < 4 \cdot |C|$. Hence there exist two vertices $x,y \in N_2(C)$ with a common neighbor $z \in N_3(C)$. Let $P^x$ ($P^y$) be a $3$-path with one endpoint $x$ ($y$) and the other endpoint in $C$. Then any $9$-vertex connected set containing $V(C) \cup V(P^x) \cup V(P^y)$ is a $(9,8^-)$-set, a contradiction. If $G[N_2[C]]$ is not unicyclic, then there exists a connected $(9,7^-)$-set containing $C$ and some other cycle, a contradiction.

\textit{Case~$3.3.$ $|C| = 6$.} We may assume that $|N_1(C)| = 6$; otherwise, $G$ contains a $5^-$-cycle, which is impossible. If $G[N_2[C]]$ is unicyclic, then $|N_2(C)| = 12$ and $|N_3(C)| \leq 2$, hence $N_2(C)$ is not independent, a contradiction. Otherwise, there exist vertices $x,y \in N_1(C)$ such that one of the following possibilities occurs: $x$ and $y$ are adjacent; $x$ and $y$ have a common neighbor in $N_2(C)$; $x$ and $y$ have adjacent neighbors in $N_2(C)$. For the first two possibilities, any $9$-vertex connected set that contains $V(C) \cup \{x,y\}$ is a $(9,8^-)$-set, a contradiction. For the last possibility, let $x',y' \in N_2(C)$ be the adjacent neighbors of $x$ and $y$, respectively. Then $V(C) \cup \{x,x',y\}$ is a $(9,8^-)$-set, again a contradiction.

\textit{Case~$3.4.$ $|C| = 7$.} We may assume that $N_1(C)$ is an independent 7-vertex set; otherwise, $G$ contains a $6^-$-cycle, which is impossible. If no two vertices from $N_1(C)$ have a common neighbor in $N_2(C)$, then $|N_2(C)| = 14$ and $|N_2[C]| = 28$, a contradiction. Otherwise, there exist vertices $x,y \in N_1(C)$ having a common neighbor $z \in N_2(C)$. Then $V(C) \cup \{x,y\}$ is a connected $(9,8^-)$-set, a contradiction.
\end{proof}

\begin{theorem}\label{t10}
$R(K_{1,4},K_{10,10}) = 29$.
\end{theorem}

\begin{proof}
The lower bound follows from Lemma~\ref{m912}(2). We prove the upper bound. Suppose, to the contrary, that there exists a $29$-vertex subcubic $(10,9)$-free graph $G$. By Lemma~\ref{edge}, we may assume that its deficient vertices form a clique. If, moreover, $G$ is disconnected, then all its cubic components have girth at most~8 by Lemma~\ref{t}(5), thus $G$ is not $(10,9)$-free by Lemma~\ref{greedy}. If $G$ is connected, but not near-cubic, then any connected $10$-vertex subset of $G$ that contains all deficient vertices is a $(10,9^-)$-set by Lemma~\ref{l2}. Suppose that $G$ is a connected near-cubic graph with a degree-$2$ vertex~$v$. Lemma~\ref{cyclic} implies that $v$ is $10$-cyclic, thus there exists a connected $(10,9^-)$-set containing $v$ together with some cycle, a contradiction.
\end{proof}

\begin{theorem}\label{t11}
$R(K_{1,4},K_{11,11}) = 31$.
\end{theorem}

\begin{proof}
The lower bound follows from Lemma~\ref{m912}(3). We prove the upper bound. Suppose that there exists a $31$-vertex subcubic $(11,9)$-free graph $G$. Observe that every connected $11$-vertex set containing at least two cycles is an $(11,9^-)$-set by Lemma~\ref{l2}.  By Lemma~\ref{edge}, we may assume that the deficient vertices of $G$ form a clique. Let $G'$ be the largest component of $G$. If $G$ is disconnected and $G'$ has at most $28$ vertices, then it is easy to check that any greedy $11$-vertex subset of $G$ is an $(11,9^-)$-set, thus $G$ is not $(11,9)$-free. If $G'$ has $29$ vertices, then $G$ has two components with deficient vertices, a contradiction. Thus we assume that $G'$ has at least $30$ vertices. If $G'$ contains a degree-$1$ vertex $x$ adjacent to a degree-$2$ vertex $y$, then the subgraph $G'[N_5[y]]$ is not acyclic, hence the vertex $y$ is $10$-cyclic and any connected $11$-vertex set containing $x$, $y$ and some cycle is an $(11,8^-)$-set, a contradiction. Otherwise, $G'$ is either a cubic $30$-vertex graph or a near-cubic $31$-vertex graph.  Let $C$ be a shortest cycle in $G'$. Lemma~\ref{t}(5) and Lemma~\ref{l1}(2) imply that either $|C| \leq 8$ or $g(G') = 9$ and the degree-$2$ vertex of $G'$ belongs to a $9$-cycle; in that case let $C$ be such a cycle.

\textit{Case 1.} $|C| \leq 6$. If $|C| \in \{3,4\}$, then, since $|N_4[C]| \leq 31$, either the subgraph $G'[N_3[C]]$ is not unicyclic or two vertices from $N_3(C)$ have a common neighbor in $N_4(C)$. Likewise, if $|C| \in \{5,6\}$, either the subgraph $G'[N_2[C]]$ is not unicyclic or two vertices from $N_2(C)$ have a common neighbor in $N_3(C)$. In both cases there exists a connected $11$-vertex set containing $C$ and some other cycle, a contradiction.

\textit{Case 2.} $|C| = 7$. If $G'[N_2[C]]$ is not unicyclic, then there exists a connected $11$-vertex set containing $C$ and another cycle, a contradiction. Otherwise, we have $|N_2[C]| \geq 25$ and $|N_2(C)| \geq 12$ (the equalities hold when $C$ has a degree-$2$ vertex). Thus $|N_3(C)| \leq 6$ and, moreover, every vertex from $N_2(C)$, except at most one, has two neighbors in $N_3(C)$, which contradicts the fact that $G'$ is subcubic.

\textit{Case 3. $|C| = 8$.} Since $g(G') = 8$, the set $N_1(C)$ is independent. If there exist two vertices $u,v \in N_1(C)$ with a common neighbor $w \in N_2(C)$, then $V(C) \cup \{u,v,w\}$ is an $(11,9^-)$-set, a contradiction. Otherwise the inequality $|N_2[C]| < 32$ implies that the set $N_1[C]$ contains a degree-2 vertex $x$. Consider two subcases.

\textit{Case 3.1.} $x \in C$. Then $|N_2[C]| = 29$ and $|N_3(C)| \leq 2$, hence there exist adjacent vertices $y,z \in N_2(C)$. Let $y'$ and $z'$ be their respective neighbors in $N_1(C)$. Then $V(C) \cup \{y',z,z'\}$ is an $(11,9^-)$-set by Lemma~\ref{l2}, a contradiction.

\textit{Case 3.2.} $x \in N_1(C)$. Then $N_2[C] = V(G')$. Consider a vertex $y \in N_2(C) \cap N_1(x)$ (it exists, because $N_1(C)$ is independent). Since $N_3(C) = \emptyset$ and $y$ has no neighbors in $N_1(C)$ other than $x$, there exist vertices $y' \in N_2(C) \cap N_1(y)$ and $x' \in N_1(C) \cap N_1(y')$. Then $V(C) \cup \{x,x',y'\}$ is an $(11,9^-)$-set by Lemma~\ref{l2}, a contradiction.

\textit{Case 4. $|C| = 9$.} Since $g(G') = 9$, the set $N_1(C)$ is independent. If two vertices $x,y \in N_1(C)$ have a common neighbor $z \in N_2(C)$, then $V(C) \cup \{x,y\}$ is an $(11,9^-)$-set by Lemma~\ref{l2}, since $C$ contains a degree-$2$ vertex. Otherwise, $|N_2[C]| = 33$, a contradiction.
\end{proof}

\begin{theorem}\label{t12}
$R(K_{1,4},K_{12,12}) = 33$.
\end{theorem}

\begin{proof}
The lower bound follows from Lemma~\ref{m912}(4). We prove the upper bound. Suppose there exists a $33$-vertex subcubic $(12,9)$-free graph $G$. Observe that every connected $12$-vertex set of $G$ containing at least two distinct cycles and a degree-$2$ vertex is a $(12,9^-)$-set by Lemma~\ref{l2}.  By Lemma~\ref{edge}, we may assume that the deficient vertices of $G$ form a clique. First, we prove the following fact.

\textit{Claim~$1$.} $G$ is a connected near-cubic graph. Suppose not; we consider two cases.

\textit{Case~$1$.} $G$ is disconnected. Let $G'$ be its largest component. If $|V(G')| \leq 30$, then it is easy to see that any greedy $12$-vertex subset of $G$ is a $(12,9^-)$-set, thus $G$ is not $(12,9)$-free. If $|V(G')| = 31$, then $G$ has two different components with deficient vertices; that contradicts Lemma~\ref{edge}. 
 Otherwise, $|V(G')| = 32$ and $G'$ is cubic. We show that $G'$ contains a connected $11$-vertex subset that contains at least two cycles; together with the isolated vertex of $G$ this set becomes a $(12,9^-)$-set by Lemma~\ref{l2}. Let $C$ be a shortest cycle in $G'$. Lemma~\ref{t}(4) implies that $|C| \leq 7$. If $|C| \in \{3,4,5,7\}$, then, for $|C| = 2q + r$, $r \in \{0,1\}$, one can easily check that the subgraph $G'[N_{5-q}[C]]$ is not unicyclic. Thus there exists an $11$-vertex connected set of $G'$ that contains $C$ and another cycle, as required. If $|C| = 6$, then $|N_3(C)| < 24$, hence it is possible to add at most $5$ vertices to $C$ and obtain a connected $11$-vertex set with another cycle, as required.

\textit{Case~$2$.} $G$ is connected, but not near-cubic. Since $|V(G)|$ is odd, we may assume, by Lemma~\ref{edge}, that $G$ has a degree-1 vertex $x$ adjacent to a degree-2 vertex $y$, while the remaining vertices have degree $3$. Let $z$ be the other neighbor of $y$. By Lemma~\ref{cyclic}, $z$ is $10$-cyclic in the connected $31$-vertex near-cubic subgraph $G - \{x,y\}$, hence $G$ has a connected $12$-vertex set that contains $x,y$ and a cycle, which is a $(12,9^-)$-set by Lemma~\ref{l2}. This contradiction completes the proof of Claim~1.

We now continue the proof of Theorem~\ref{t12}. Let $v$ be the only degree-2 vertex of the near-cubic graph $G$. 

\textit{Case~$1$.} $v$ is 3-cyclic. Then $v$ belongs to some triangle $C$. The subgraph $G[N_4[C]]$ cannot be unicyclic, hence there exists a $12$-vertex connected set containing $C$ and another cycle, which is a $(12,9^-)$-set.

\textit{Case~$2$.} $v$ is 4-cyclic, but not 3-cyclic.

\textit{Case~$2.1$.} $v$ belongs to a $4$-cycle $C$. The subgraph $G[N_4[C]]$ cannot be unicyclic, hence there exists a $12$-vertex connected set containing $C$ and another cycle, which is a $(12,9^-)$-set.

\textit{Case~$2.2$.} There exists a triangle $C$ such that $v \in N_1(C)$. The subgraph $G[N_4[C]]$ cannot be unicyclic, hence there exists a connected $11$-vertex set $U$ containing $C$ and some other cycle. Then either $U \cup \{v\}$ is a $(12,9^-)$-set or it can be extended to a $(12,9^-)$-set, a contradiction.

\textit{Case~$3.$} $v$ is 5-cyclic, but not 4-cyclic.

\textit{Case~$3.1.$} $v$ belongs to a $5$-cycle $C$. The subgraph $G[N_3[C]]$ cannot be unicyclic, hence there exists a $12$-vertex connected set containing $C$ and another cycle, which is a $(12,9^-)$-set.

\textit{Case~$3.2.$} $v$ is adjacent to a vertex of a 4-cycle $C$. If $G[N_3[C]]$ is not unicyclic, there exists a $(12,9^-)$-set containing $v$ and two cycles, including $C$. If $G[N_3[C]]$ is unicyclic, then $|N_3(C)| = 14$ and $|N_4(C)| \leq 4$, which is not possible, since $N_3(C)$ is independent and each of its vertices has two neighbors in $N_4(C)$.

\textit{Case~$3.3.$} There exists a triangle $C$ such that $v \in N_2(C)$. Let $v'$ be the neighbor of $v$ in $N_1(C)$. If $G[N_3[C]]$ is not unicyclic, there exists a $(12,9^-)$-set containing both $v,v'$ and two cycles including $C$. Otherwise, $|N_3[C]| = 23$ and $|N_4(C)| \leq 10$, hence there exist vertices $x,y \in N_3(C)$ with a common neighbor $z \in N_4(C)$. Let $P^x$ ($P^y$) be a $4$-path with one endpoint $x$ ($y$) and the other endpoint in $C$. Then any $12$-vertex connected set containing $V(C) \cup V(P^x) \cup V(P^y) \cup \{v,v',x,y,z\}$ is a $(12,9^-)$-set, a contradiction.


\textit{Case~$4.$} $v$ is not 5-cyclic. If $v$ belongs to a $6$-cycle $C$, then $G[N_3[C]]$ is not unicyclic, hence there exists a connected $(12,9^-)$-set containing $C$ and another cycle, a contradiction. Otherwise $|N_3[v]| = 15$. We now show that $|N_4(v)| \geq 14$.

Suppose that $|N_4(v)| \leq 13$. If $G[N_3(v)]$ contains at least two edges, $x_1y_1$ and $x_2y_2$ (not necessarily disjoint), then any connected $12$-vertex set containing $N_2[v] \cup \{x_1,y_1,x_2,y_2\}$ is a $(12,9^-)$-set, as it contains two distinct cycles. If $G[N_3(v)]$ contains exactly one edge $xy$, then there exist (not necessarily distinct from $x$ and $y$) vertices $w',w'' \in N_3(v)$ with a common neighbor $w''' \in N_4(v)$. Any connected $12$-vertex set containing $N_2[v] \cup \{x,y,w',w'',w'''\}$ is a $(12,9^-)$-set. Finally, if $G[N_3(v)]$ is empty, then there exist three pairs of vertices, $x_i,y_i \in N_3(v)$ with a common neighbor $z_i \in N_4(v)$. If, say, $x_1 = x_2$, then the set $N_2[v] \cup \{x_1,y_1,y_2,z_1,z_2\}$ contains two distinct cycles and thus every connected $12$-vertex set containing it is a $(12,9^-)$-set. If all six vertices $x_i$ and $y_i$ are distinct, then there exist two of them (by symmetry, $x_1$ and $x_2$, or $x_1$ and $y_1$) with a common neighbor in $N_2(v)$. Hence there exists a vertex $z \in N_2(v)$ such that $\{x_1,x_2,y_1,y_2\} \cap N(z) = \emptyset$. Then $(N_2[v] \setminus \{z\}) \cup \{x_1,y_1,x_2,y_2,z_1,z_2\}$ is a $(12,9^-)$-set, a contradiction. 

Therefore, $|N_4(v)| \geq 14$ and $|N_5(v)| \leq 4$. Thus there exists a vertex $x \in N_4(v)$ with no neighbors in $N_5(v)$. Construct a set $U$ containing a $7$-vertex set $N_1[\{v,x\}]$ and a shortest path from every vertex of $N_1(x)$ to $N_1(v)$. Since at least one vertex from $N_1(x)$ belongs to $N_3(v)$, these shortest paths have at most $5$ total vertices not from $N_1[\{v,x\}]$, thus $U$ contains at most $12$ vertices and any connected $12$-vertex set containing $U$ is a $(12,9^-)$-set by Lemma~\ref{l2} (take $d_2 = 1$ and $s \geq 2$). Hence $G$ is not $(12,9)$-free, a final contradiction.

\end{proof}

\section*{Statements and declarations}

\subsection*{Data availability} The graph encodings used to establish the computational lower bounds are provided in Table~\ref{table1} in \verb|sparse6| format.

\subsection*{Use of artificial intelligence} The authors used AI tools to obtain the witness graphs for the lower bounds, simplify the proofs and to improve the readability of the text. The authors take full responsibility for the accuracy and originality of the manuscript.



\begin{thebibliography}{99}

\bibitem{Radz}
S.~P. Radziszowski,
\emph{Small Ramsey numbers},
Electron. J. Combin., Dynamic Survey DS1, revision \#18 (2026).

\bibitem{Par}
T.~D. Parsons,
\emph{Ramsey graphs and block designs, I},
Trans. Amer. Math. Soc. \textbf{209} (1975), 33--44.

\bibitem{Zhang}
X.~Zhang, Y.~Chen, and T.~C. Edwin Cheng,
\emph{Some values of Ramsey numbers for $C_4$ versus stars},
Finite Fields Appl. \textbf{45} (2017), 73--85.

\bibitem{Boza}
L.~Boza,
\emph{Exact values and bounds for Ramsey numbers of $C_4$ versus a star graph},
\url{https://arxiv.org/abs/2409.12770} (2024)

\bibitem{BFRS}
S.~A. Burr, R.~J. Faudree, C.~C. Rousseau, and R.~H. Schelp,
\emph{On Ramsey numbers involving starlike multipartite graphs},
J. Graph Theory \textbf{7} (1983), 395--409.

\bibitem{Stev}
S.~Stevens,
\emph{Ramsey numbers for stars versus complete multipartite graphs},
Congr. Numer. \textbf{73} (1990), 63--71.

\bibitem{Eze}
L.~C. Eze, R.~Jajcay, T.~Jajcayov\'a, and D.~Z\'avack\'a,
\emph{Theoretical and computational approaches to determining sets of orders for $(k,g)$-graphs},
Discrete Appl. Math. \textbf{380} (2026), 572--587.

\bibitem{Har1}
F.~Harary,
\emph{Recent results on generalized Ramsey theory for graphs},
in: Graph Theory and Applications (Y.~Alavi et al., eds.),
Lecture Notes in Math., vol.~303, Springer, Berlin, 1972, pp.~125--138.

\bibitem{LoM4}
R.~Lortz and I.~Mengersen,
\emph{Further Ramsey numbers for small complete bipartite graphs},
Ars Combin. \textbf{79} (2006), 195--203.

\bibitem{Shao}
Z.~Shao, J.~Xu, Q.~Bao, and L.~Pan,
\emph{Computation of some generalized Ramsey numbers},
J. Combin. Math. Combin. Comput. \textbf{75} (2010), 217--228.

\end{thebibliography}
\end{document}